\documentclass{amsart}
\usepackage{amsfonts}
\usepackage{amsmath,amssymb}
\usepackage{amsthm}
\usepackage{amscd}
\usepackage{graphics}
\usepackage{graphicx}

\theoremstyle{remark}{
\newtheorem{Def}{{\rm Definition}}
\newtheorem{Ex}{{\rm Example}}

\newtheorem{Prob}{{\rm Problem}}

}
\theoremstyle{plain}
{

\newtheorem{Prop}{Proposition}
\newtheorem{Thm}{Theorem}

}

\begin{document}
\title[Reeb spaces of functions associated to two smooth functions]{Fundamental examples of Reeb spaces of smooth functions defined from two graphs of smooth functions with same asymptotic behaviors}
\author{Naoki kitazawa}

\address{Osaka Central Advanced Mathematical Institute (OCAMI) \\
3-3-138 Sugimoto, Sumiyoshi-ku Osaka 558-8585
TEL: +81-6-6605-3103
}
\email{naokikitazawa.formath@gmail.com}
\urladdr{https://naokikitazawa.github.io/NaokiKitazawa.html}
\maketitle
\begin{abstract}

{\it Reeb spaces} of (continuous) real-valued functions on (nice) topological spaces are the spaces whose underlying sets consist of all connected components ({\it contours}) of their level sets and seen naturally as quotient spaces of the spaces. They are "$1$-dimensional" spaces in various nice cases. They are graphs or graphs with {\it ends} for smooth function cases with nice singularities and behaviors. Reeb spaces have been fundamental and important in theory of Morse functions and more general smooth functions and applications to geometry, since the 20th century.

We present Reeb spaces homeomorphic to infinite graphs (with ends) for functions on non-compact manifolds with no boundary. This paper is a note on cases previously obtained by the author. More explicitly, we consider a natural smooth map onto the region surrounded by the graphs of two smooth real-valued functions in the plane and its composition with the canonical projection.

\end{abstract}
%【REVISE】 combinatoric ～ is → combinatorial object. It is .
%【REVISE】  such that a point is a vertex if and only if the corresponding connected component of the level set contains some singular points → whose vertex set is the set of all points containing some singular points in the corresponding connected component of the level set .
%【REVISE】 We delete "extending the result before".
\section{Introduction.}
\label{sec:1}

The {\it Reeb space} $R_c$ of a (continuous) real-valued function $c:X \rightarrow \mathbb{R}$ on a (nice) topological space $X$ is regarded to be the space of all connected components ({\it contours}) of all preimages ({\it level sets}) $c^{-1}(q)$. In the case of differentiable functions, we add "{\it regular}" (resp. "{\it critical}") before "contour" and "level set" if it contains no critical point (resp. some critical points) of the function. We explain the notions mathematically and rigorously. We can induce the equivalence relation ${\sim}_c$ on $X$ by the following rule: the relation $x_1 {\sim}_c x_2$ holds if and only if $x_1$ and $x_2$ are in a same contour of a same level set $c^{-1}(y)$. The quotient map $q_c:X \rightarrow R_c:=X/{{\sim}_c}$ and the unique continuous function $\bar{c}:R_c \rightarrow \mathbb{R}$ with $c=\bar{c} \circ q_c$ is also defined. 

These tools or objects have been well-known as fundamental and important tools or objects in theory of Morse functions and more general differentiable functions with nice singularity (\cite{reeb}).
$R_c$ is, for a function $c$ which is not so wild mathematically, has the structure of a graph. For the Morse function case and Morse-Bott function case (on compact manifolds), see \cite{izar, martinezalfaromezasarmientooliveira} and for more general cases, see \cite{gelbukh, saeki1, saeki2}, for example. Especially, \cite[Theorem 7.5]{gelbukh}, \cite[Theorem 3.1]{saeki1}, and \cite[Theorems 2.1 and 2.8 and "2 Reeb space and its graph structure"]{saeki2} are important as theory on topologies and graph structures of these spaces. 

Hereafter, let ${\mathbb{R}}^k$ denote the $k$-dimensional {\it Euclidean} ({\it real affine}) {\it space} ($\mathbb{R}:={\mathbb{R}}^1$). Let ${\pi}_{k_1,k_2}:{\mathbb{R}}^{k_1} \rightarrow {\mathbb{R}}^{k_2}$ denote the canonical projection mapping $x=(x_1,x_2) \in {\mathbb{R}}^{k_1-k_2} \times {\mathbb{R}}^{k_2}$ to $x_1$ with $k_1>k_2 \geq 1$. A {\it graph} means a $1$-dimensional connected, locally finite, and locally compact CW complex such that the closure of each $1$-cell ({\it edge}) is homeomorphic to the $1$-dimensional disk $D^1:=\{x \in \mathbb{R} \mid x^2 \leq 1\}$: $0$-cell is a {\it vertex}. A {\it graph with ends} (with related notions) is defined by weakening the condition on edges as follows: the closure of each $1$-cell ({\it edge}) is homeomorphic to $D^1$ or a ray $\{x \geq 0 \mid x \in \mathbb{R}\}$. For CW complexes, see \cite{hatcher} for example.   

In the case of the smooth manifold $X$ with no boundary, $R_c$ is, in a certain nice case, a graph (with ends) whose vertex set consists of points for all critical contours of $c$. This is the so-called {\it Reeb graph} of $c$ and in most studies including presented studies above, this definition is adopted. We can see the Reeb graph as a digraph (with ends) by giving an orientation on each edge by the following rule: an edge $e$ is an oriented edge departing from (resp. entering) $v_{e}$ if the restriction of $\bar{c}$ to the closure of $e$ in the graph has the minimum (resp, maximum) at $v_e$. 
\begin{Prob}
\label{prob:1}
Can we construct a (nice) smooth function whose
Reeb space is regarded to be the Reeb graph isomorphic to a given graph (resp. with ends)?
\end{Prob}

This is a kind of fundamental, natural and incredibly new problems. Sharko has launched in this 2006 (\cite{sharko}). This is a study on smooth functions on closed surfaces whose critical points $p$ are of the form $c(z)={\Re} z^{l}+t_p$ or critical points of Morse functions and whose Reeb graphs are given graphs of a certain class. Here $\mathbb{C}$ is the space of all complex numbers, $l>1$ is an integer, $t_p \in \mathbb{R}$, and ${\Re} s$ denotes the real part of $s \in \mathbb{C}$. This is extended to arbitrary finite graphs in \cite{masumotosaeki} and functions whose critical points may not be isolated are studied. In \cite{michalak}, a certain class of finite graphs is important and Morse functions on closed manifolds with prescribed Reeb graphs of the class and regular contours which are spheres are reconstructed. 
The author has contributed to this, in, \cite{kitazawa1, kitazawa4}, for example. There the author respects not only graphs, but also shapes of level sets. In addition, in \cite{kitazawa2}, the author considers smooth functions on non-compact manifolds and so-called {\it non-proper} functions, explicitly, first. As another important pioneering study by the author, \cite{kitazawa3} is on real algebraic construction with prescribed Reeb graphs. The author is contributing to real algebraic construction, mainly, recently, where we ignore related knowledge, and omit precise exposition on other related papers and preprints, except \cite{kitazawa5, kitazawa6}. Later, as a kind of our main ingredients, related arguments appear in our main ingredients and we discuss the arguments in self-contained ways (Theorem \ref{thm:1}).

Recently the author has started to consider infinite graphs (with ends) as Reeb spaces and obtained several explicit cases. They are also constructed as, fully or densely, real algebraic or real analytic and based on the paper \cite{kitazawa3} and the preprints \cite{kitazawa5, kitazawa6}. Related preprints are \cite{kitazawa7, kitazawa8}. We ignore related non-trivial knowledge and discuss related arguments in self-contained ways. The present paper is also a kind of our new remark on examples of \cite{kitazawa7, kitazawa8}, especially \cite[Theorem 1]{kitazawa7}. There real algebraic or real analytic functions or maps are presented and answers to Problem \ref{prob:1} are presented as the restrictions of the canonical projections ${\pi}_{k_1,k_2}$ to the zero sets of these nice smooth maps between the real affine spaces. We respect these examples and present answers from a kind of essentially generalized viewpoints.

In the next section, we prove our new result. This is for a generalization of the case presented in \cite[Theorem 1 and problem 6]{kitazawa7} and in Example \ref{ex:1}, this is explained again.  Explicitly, we consider a natural smooth map onto the region surrounded by the graphs of two smooth real-valued functions $c_1$ and $c_2$ diverging to a same value at each infinity in ${\mathbb{R}}^2$, obtained via Theorem \ref{thm:1}, and its composition with the canonical projection ${\pi}_{2,1}$.  We investigate the Reeb space of this function. 
\section{Our main result.}
We introduce or review several notions additionally. Some already appear in the last section, where it is no problem. They are fundamental and well-known.  

For a map $c:X \rightarrow Y$ and a subset $Z \subset Y$, let $c {\mid}_Z$ denote the restriction of $c$ to $Z$.

A {\it proper} map $c:X \rightarrow Y$ between topological spaces $X$ and $Y$ is a map whose preimage $c^{-1}(K)$ is compact for any each subset $K \subset X$. A {\it non-proper} map is one which may not be proper.

For a differentiable map between differentiable manifolds, a point in the manifold of the domain is a {\it singular} point of it where the rank of the differential of the map is smaller than both the dimension of the manifold of the domain and that of the manifold of the target. 
The {\it singular set} $S(c)$ of the map $c:X \rightarrow Y$ is defined as the set of all singular points of the map. The {\it singular value set} of the map $c$ is the image $c(S(c))$ of the singular set of the map.
We use "{\it critical}" in the case where the map is the real-valued map ($Y:=\mathbb{R}$).

Theorem \ref{thm:1} is essentially from \cite{kitazawa3}. 
\begin{Thm}
	\label{thm:1}
	Let $m \geq 2$ be an integer.
	Let $c_i:\mathbb{R} \rightarrow \mathbb{R}$ {\rm (}$i=1,2${\rm )} be smooth functions such that $c_1(x)<c_2(x)$ for any $x \in \mathbb{R}$.
	Then consider a region $D_{c_1,c_2}:=\{(x_1,x_2)\mid c_1(x_2)<x_1<c_2(x_2)\}$ and its closure $\overline{D_{c_1,c_2}}$ in ${\mathbb{R}}^2$.
	Then we have an $m$-dimensional smooth submanifold $X_{m,c_1,c_2}:=\{(x_1,x_2,{(y_j)}_{j=1}^{m-1}) \in \overline{D_{c_1,c_2}} \times {\mathbb{R}}^{m-1} \subset {\mathbb{R}}^{m+1} \mid (x_1-c_1(x_2))(c_2(x_2)-x_1)-{\Sigma}_{j=1}^{m-1} {y_j}^2=0\}$ with no boundary and the restriction ${\pi}_{m+1,2} {\mid}_{X_{m,c_1,c_2}}$ is a surjection onto $\overline{D_{c_1,c_2}}$ whose singular set is $\{(x_1,x_2,{(y_j)}_{j=1}^{m-1}) \mid (x_1,x_2) \in \overline{D_{c_1,c_2}}-D_{c_1,c_2}\}=\{(x_1,x_2,{(0)}_{j=1}^{m-1}) \mid (x_1,x_2) \in \overline{D_{c_1,c_2}}-D_{c_1,c_2}\}$.
	\end{Thm}
\begin{proof}
	This is a main ingredient of \cite{kitazawa3} and we prove this in a self-contained way.
	It is a main ingredient to show that by implicit function theorem, this is the zero set of the smooth function and is a smooth manifold with no boundary. This is the zero set of the smooth function $(x_1-c_1(x_2))(c_2(x_2)-x_1)-{\Sigma}_{j=1}^{m-1} {y_j}^2$ by the definition.
	
	In the case $(x_1,x_2) \in D_{c_1,c_2}$ in the zero set, the value of the partial derivative of the function $(x_1-c_1(x_2))(c_2(x_2)-x_1)-{\Sigma}_{j=1}^{m-1} {y_j}^2$ by some $y_j$ is not $0$.
	
	In the case $(x_1,x_2) \in \overline{D_{c_1,c_2}}-D_{c_1,c_2}$ in the zero set, the value of the partial derivative of the function $(x_1-c_1(x_2))(c_2(x_2)-x_1)-{\Sigma}_{j=1}^{m-1} {y_j}^2$ by $x_1$ is $c_1(x_2)-x_1$ in the case $c_2(x_2)-x_1=0$ and $c_2(x_2)-x_1$ in the case $x_1-c_1(x_2)=0$. The values are both non-zero.

This means that the ranks of the differential of the function are always $1$ on the zero set. The zero set is an $m$-dimensional smooth submanifold of ${\mathbb{R}}^{m+1}$ with no boundary. We can prove the remaining easily by the situation. 
	\end{proof}
	\begin{Thm}
		\label{thm:2}
		In Theorem \ref{thm:1}, the following hold. We abuse the notation.
		\begin{enumerate}
			\item \label{thm:2.1} A point of $X_{m,c_1,c_2}$ is a critical point of the function ${\pi}_{m+1,1} {\mid}_{X_{m,c_1,c_2}}$ if and only if it is a singular point of the surjection ${\pi}_{m+1,2} {\mid}_{X_{m,c_1,c_2}}$ onto
			$\overline{{D}_{m,c_1,c_2}}$ and of the form $(c_i(s),s,{(0)}_{j=1}^{m-1})$, where $s$ is a critical point of the function $c_i$.
			
			\item \label{thm:2.2} Suppose that the intersection $\{\mathbb{R}-(\{q_1\}\bigcup \{q_2\})\}
 \bigcap (c_1(S(c_1)) \bigcup c_2(S(c_2)))$ is a discrete and closed subset of $\{\mathbb{R}-(\{q_1\}\bigcup \{q_2\})\}$ and that at least one function of $c_1$ and $c_2$ has at least one critical point. Suppose also the following.
			
			\begin{enumerate}
				\item \label{thm:2.2.1}
				As $t \in \mathbb{R}$ diverges to $-\infty$, then both the values $c_1(t)$ and $c_2(t)$ converge to $q_1 \in \mathbb{R}$.
				\item \label{thm:2.2.2} As $t \in \mathbb{R}$ diverges to $+\infty$, then both the values $c_1(t)$ and $c_2(t)$ converge to $q_2 \in \mathbb{R}$.
				
				\item \label{thm:2.2.3} For $i=1,2$, each connected component $K_{c_1,c_2,q_i,j}$ of the set ${{\pi}_{m+1,1}}^{-1}(q_i) \bigcap X_{m,c_1,c_2}$ has a small compact and connected neighborhood $N_{c_1,c_2,q_i,j}$ in $X_{m,c_1,c_2}$ consisting of contours of the the restriction ${\pi}_{m+1,1} {\mid}_{X_{m,c_1,c_2}}$ and exactly one contour  in a level set of the form ${{\pi}_{m+1,1}}^{-1}(q_1) \bigcap X_{m,c_1,c_2}$ or ${{\pi}_{m+1,1}}^{-1}(q_2) \bigcap X_{m,c_1,c_2}$. The complementary set of $K_{c_1,c_2,q_i,j}$ in the connected component $N_{c_1,c_2,q_i,j}$ consists of finitely many connected sets and contains no critical contour of ${\pi}_{m+1,1} {\mid}_{X_{m,c_1,c_2}}$.
				
			\end{enumerate}
		Then the Reeb space of the function ${\pi}_{m+1,1} {\mid}_{X_{m,c_1,c_2}}$ is regarded to be the Reeb graph.
		\end{enumerate}
		\end{Thm}
		\begin{proof}
			The statement (\ref{thm:2.1}) holds from our construction. 
			
			We prove the statement (\ref{thm:2.2}). First, the critical sets of the functions $c_1$ and $c_2$ are discrete by the assumption. From the conditions (\ref{thm:2.2.1}, \ref{thm:2.2.2}) and the structures of the maps and the manifolds, the restriction of ${\pi}_{m+1,1} {\mid}_{X_{m,c_1,c_2}}$ to the preimage of $\mathbb{R}-(\{q_1\}\bigcup \{q_2\})$ is proper. The intersection $\{\mathbb{R}-(\{q_1\}\bigcup \{q_2\})\}
			\bigcap (c_1(S(c_1)) \bigcup c_2(S(c_2)))$ is a discrete and closed subset of $\{\mathbb{R}-(\{q_1\}\bigcup \{q_2\})\}$. By virtue of \cite{saeki1, saeki2}, its Reeb space is regarded to be homeomorphic to a $1$-dimensional manifold or a graph which is not homeomorphic to a manifold. In addition, if at least one critical point of the function exists, then the Reeb space of the restriction of ${\pi}_{m+1,1} {\mid}_{X_{m,c_1,c_2}}$ to the preimage of $\mathbb{R}-(\{q_1\}\bigcup \{q_2\})$ is regarded to be the Reeb graph. For \cite{saeki1, saeki2}, related to this argument, mainly, consult \cite[Theorem 3.1]{saeki1} and \cite[Theorems 2.1 and 2.8 and "2 Reeb space and its graph structure"]{saeki2}. Related to this theory, see also \cite{gelbukh}. We consider the condition (\ref{thm:2.2.3}). According to this, around each point in ${\bar{{\pi}_{m+1,1} {\mid}_{X_{m,c_1,c_2}}}}^{-1}(\{q_1\}\bigcup \{q_2\})$, the Reeb space is homeomorphic to either a $1$-dimensional manifold or a finite graph which is not homeomorphic to a manifold. This is also a Hausdorff space by the assumption that $N_{c_1,c_2,q_i,j}$ in $X_{m,c_1,c_2}$ contains exactly one contour of ${\pi}_{m+1,1} {\mid}_{X_{m,c_1,c_2}}$ in level sets of the form ${{\pi}_{m+1,1}}^{-1}(q_1) \bigcap X_{m,c_1,c_2}$ or ${{\pi}_{m+1,1}}^{-1}(q_2) \bigcap X_{m,c_1,c_2}$. The Reeb space is also regarded to be the Reeb graph (if there exists a critical point of the function ${\pi}_{m+1,1} {\mid}_{X_{m,c_1,c_2}}$ there), locally. This also comes from the general theory above. We can easily see that two distinct contours of the function ${\pi}_{m+1,1} {\mid}_{X_{m,c_1,c_2}}$ are always separated by suitably chosen neighborhoods in $X_{m,c_1,c_2}$ and the Reeb space of the function ${\pi}_{m+1,1} {\mid}_{X_{m,c_1,c_2}}$ is a Hausdorff space. It is also a locally connected space of course.

At least one of the functions $c_1$ and $c_2$ has at least one critical point, by the assumption. From this together with the arguments above, the Reeb space of the function ${\pi}_{m+1,1} {\mid}_{X_{m,c_1,c_2}}$ is regarded to be the Reeb graph and a graph with ends. 
			
			This completes the proof of (\ref{thm:2.2}).
			
		\end{proof}
		\begin{Prop}
			\label{prop:0}
			In Theorem \ref{thm:1}, if $S(c_1)$ and $S(c_2)$ are discrete and conditions of {\rm (}\ref{thm:2.2.1}, \ref{thm:2.2.2}{\rm )} of Theorem \ref{thm:2} are satisfied, then the intersection $\{\mathbb{R}-(\{q_1\}\bigcup \{q_2\})\}
			\bigcap (c_1(S(c_1)) \bigcup c_2(S(c_2)))$ is a discrete and closed subset of $\mathbb{R}-(\{q_1\}\bigcup \{q_2\})$.

		\end{Prop}
		\begin{proof}
			A compact subset $K_{c_1,c_2}$ in
		$\mathbb{R}-(\{q_1\}\bigcup \{q_2\})$ is also bounded and closed in $\mathbb{R}-(\{q_1\}\bigcup \{q_2\})$. In the preimage ${c_1}^{-1}({K_{c_1,c_2}}) \bigcup {c_2}^{-1}({K_{c_1,c_2}})$, by the discreteness of the critical set of $c_i$ and the boundedness and compactness of ${c_1}^{-1}({K_{c_1,c_2}}) \bigcup {c_2}^{-1}({K_{c_1,c_2}})$, $(S(c_1) \bigcup S(c_2)) \bigcap ({c_1}^{-1}({K_{c_1,c_2}}) \bigcup {c_2}^{-1}({K_{c_1,c_2}}))$ is finite and $(c_1(S(c_1)) \bigcup c_2(S(c_2))) \bigcap K_{c_1,c_2}$ is discrete in $K_{c_1,c_2}$. For each point $p$ of $\mathbb{R}-(\{q_1\}\bigcup \{q_2\})$, we have its compact neighborhood $K_{p,c_1,c_2} \ni p$ in $\mathbb{R}-(\{q_1\}\bigcup \{q_2\})$. From this, we can see that $(\mathbb{R}-(\{q_1\}\bigcup \{q_2\})) \bigcap (c_1(S(c_1)) \bigcup c_2(S(c_2)))$ is a discrete and closed subset in $(\mathbb{R}-(\{q_1\}\bigcup \{q_2\}))
		\bigcap (c_1(S(c_1)) \bigcup c_2(S(c_2)))$.
		\end{proof}
%			Note that this is a part of the assumption of Theorem \ref{thm:2} {\rm (}\ref{thm:2.2}{\rm )}. This holds for the case as in Theorem \ref{thm:4} and \ref{thm:5}, presented later.
		
For such functions $c_1$ and $c_2$ of Proposition \ref{prop:0}, Morse functions and real analytic functions are important.
		
	The following class of smooth real-valued functions on $\mathbb{R}$ is a generalized class of the so-called
	{\it unimodal} real-valued functions. 
\begin{Def}
	A smooth function $c:\mathbb{R} \rightarrow \mathbb{R}$ is {\it almost-unimodal} if the following hold.
	\begin{itemize}
		\item $c$ has the maximum $c_{\rm M}$.
		\item There exists an interval  $\{t \mid a_0 \leq t \leq b_0\}$ and $c$ has the maximum $c_{\rm M}$ there. Furthermore, there exist a sequence $\{a_j\}_{j=1}^{\infty}$ such that $a_{j}<a_{j+1}$ for each $j \geq 0$ and that as $j$ diverges to $+\infty$, $a_j$ diverges to $+\infty$ and another sequence  $\{b_j\}_{j=1}^{\infty}$ such that $b_{j+1}<b_j$ for each $j \geq 0$ and that as $j$ diverges to $+\infty$, $b_j$ diverges to $-\infty$ enjoying the following.
		\begin{itemize}
			\item Let $c_{{\rm M},a,j}$ be the maximum of $c$ on the interval $\{t \mid a_j \leq t \leq a_{j+1}\}$. Then $\{c_{{\rm M},a,j}\}_{j=1}^{\infty}$ is a sequence such that $c_{{\rm M},a,j+1} \leq c_{{\rm M},a,j}$ for each $j \geq 0$. 
			\item Let $c_{{\rm M},b,j}$ be the maximum of $c$ on the interval $\{t \mid b_{j+1} \leq t \leq b_{j}\}$. Then $\{c_{{\rm M},b,j}\}_{j=1}^{\infty}$ is a sequence such that $c_{{\rm M},b,j+1} \leq c_{{\rm M},b,j}$ for each $j \geq 0$.
		\end{itemize} 
		
	\end{itemize}
	\end{Def}
Proposition \ref{prop:1} is seen as a part of our main result.
\begin{Prop}
	\label{prop:1}
	Let $c_1$ and $c_2$ be functions of Theorem \ref{thm:2} {\rm (}\ref{thm:2.2}{\rm )}. We abuse the notation. Suppose also the following.
	\begin{enumerate}
		\item \label{prop:1.1} The critical set of $c_1$ is unbounded or the function $-c_1$ is almost unimodal.

			\item \label{prop:1.2} The critical set of $c_2$ is unbounded or the function $c_2$ is almost unimodal.

	\end{enumerate}
Then, the Reeb space of the function ${\pi}_{m+1,1} {\mid}_{X_{m,c_1,c_2}}$ is regarded to be the Reeb graph and isomorphic to a finite or infinite graph.
\end{Prop}
\begin{proof}
We choose a point $p$ of $X_{m,c_1,c_2}$ mapped to a point $p_e$ in some edge $e$ of the Reeb graph by the quotient map and a point $p_{D}:={\pi}_{2,1}(p) \in \overline{{D}_{m,c_1,c_2}}$. Let $p_{D}:=(p_{D,1},p_{D,2})$. We consider the point $p_{D,c_2}:=(c_2(p_{D,2}),p_{D,2})$. The set $\{(t,p_{D,2}) \mid p_{D,1} \leq t \leq c_2(p_{D,2})\}$ is in $\overline{{D}_{m,c_1,c_2}}$. Due to the assumption (\ref{prop:1.2}), we can consider some path from $p_{D,c_2}$ to another point $p_{s,c_2}$ in $\overline{{D}_{m,c_1,c_2}}-{D}_{m,c_1,c_2}$ and have the path departing from $p_{D}$ obtained by joining the set $\{(t,p_{D,2}) \mid p_{D,1} \leq t \leq c_2(p_{D,2})\}$, which is a one-point set or a segment, and the path in $\overline{{D}_{m,c_1,c_2}}-{D}_{m,c_1,c_2}$ in such a way that the restriction of ${\pi}_{2,1}$ to the resulting path is injective and that the point $(p_{s,c_2},{(0)}_{j=1}^{m-1}) \in {{\mathbb{R}}}^2 \times {\mathbb{R}}^{m-1}={\mathbb{R}}^{m+1}$ is a critical point of ${\pi}_{m+1,1} {\mid}_{X_{m,c_1,c_2}}$. 
Related to the assumption (\ref{prop:1.1}), we can argue similarly. 
We can see that the closure of $e$ in the Reeb graph is homeomorphic to $D^1$. More precisely, from the argument above, we naturally have a union of finitely many closures of edges of the Reeb graph with the union containing $p_e$.
\end{proof}
We explain some related examples.
\begin{Ex}\label{ex:1}
\begin{enumerate}
\item \label{ex:1.1} The case $c_1(x):=\frac{1}{2(x^2+1)} {\sin}^2 x$ and $c_2(x):=\frac{1}{x^2+1}$ is an explicit case of Proposition \ref{prop:1} with the critical set of $c_1$ being unbounded, $c_2$ being (almost) unimodal, and $(q_1,q_2)=(0,0)$. This is also from \cite[Theorem 1]{kitazawa7}.
\item \label{ex:1.2} The case $c_1(x):=-\frac{1}{2(x^2+1)} {\sin}^2 x$ and $c_2(x):=\frac{1}{x^2+1}$ seems to be an explicit case of Proposition \ref{prop:1} with the critical set of $c_1$ being unbounded, $c_2$ being almost unimodal, and $(q_1,q_2)=(0,0)$. However, the Reeb space is not homeomorphic to any graph with ends. Indeed, the condition (\ref{thm:2.2.3}) from Theorem \ref{thm:2} is dropped and the set ${{\pi}_{2,1}}^{-1}(q_i) \bigcap \overline{D_{c_1,c_2}}$ and ${{\pi}_{m+1,1}}^{-1}(q_i) \bigcap X_{m,c_1,c_2}$ are connected and non-compact. the value of the quotient map $q_{{\pi}_{m+1,1} {\mid}_{X_{m,c_1,c_2}}}:X_{m,c_1,c_2} \rightarrow R_{{\pi}_{m+1,1} {\mid}_{X_{m,c_1,c_2}}}$ at points in the set ${{\pi}_{m+1,1}}^{-1}(q_i) \bigcap X_{m,c_1,c_2}$, which is an unique element, is incident to infinitely many copies of $1$-cells, in the Reeb space. The Reeb space has the structure of a $1$-dimensional CW complex which is not locally finite or locally compact.
\end{enumerate}
	\end{Ex}
Hereafter, a real-valued smooth function $c_{I,\mathbb{R}}:I \rightarrow \mathbb{R}$ on $I \subset \mathbb{R}$, is {\it strictly increasing} ({\it decreasing}) if for any pair $(t_1,t_2)$ of real numbers $t_1<t_2$ in $I$, $c_{I,\mathbb{R}}(t_1)<c_{I,\mathbb{R}}(t_2)$ (resp. $c_{I,\mathbb{R}}(t_1)>c_{I,\mathbb{R}}(t_2)$) always holds.

In Proposition \ref{prop:1}, if the critical set of $c_2$ is bounded below and $c_2$ is almost-unimodal, then the restriction of $c_2$ to some interval $\{t \leq b_{\rm m}\} \subset \mathbb{R}$ is strongly increasing, and if the critical set of $c_2$ is bounded above and $c_2$ is almost-unimodal, then the restriction of $c_2$ to some interval $\{t \geq b_{\rm M}\} \subset \mathbb{R}$ is strongly decreasing. There if the critical set of $c_1$ is bounded below and $-c_1$ is almost-unimodal, then the restriction of $-c_1$ to some interval $\{t \leq b_{\rm m}\} \subset \mathbb{R}$ is strongly increasing (for distinct  numbers $t_1<t_2$ in the interval $-c_1(t_1)<-c_1(t_2)$), and if the critical set of $c_1$ is bounded above and $-c_1$ is almost-unimodal, then the restriction of $-c_1$ to some interval $\{t \geq b_{\rm M}\} \subset \mathbb{R}$ is strongly decreasing (for distinct  numbers $t_1<t_2$ in the interval $-c_1(t_1)>-c_1(t_2)$).

Related to this, we present two propositions. 

\begin{Prop}
	\label{prop:2}
	Let $c_1$ and $c_2$ be functions of Theorem \ref{thm:2} {\rm (}\ref{thm:2.2}{\rm )}. We abuse the notation. Suppose also that all of the following four hold.
	\begin{enumerate}
		\item \label{prop:2.1} The critical set of $c_2$ is unbounded below, or the restriction of $c_2$ to some interval of the form $\{t \leq b_{{\rm m},0}\} \subset \mathbb{R}$ is strongly increasing. 
\item  \label{prop:2.2} The critical set of $c_2$ is unbounded above, or the restriction of $c_2$ to some interval of the form $\{t \geq b_{{\rm M},0}\} \subset \mathbb{R}$ is strongly decreasing. 
		\item \label{prop:2.3} The critical set of $c_1$ is unbounded below, or the restriction of $-c_1$ to some interval of the form $\{t \leq b_{{\rm m},0}\} \subset \mathbb{R}$ is strongly increasing.
	
			\item \label{prop:2.4} The critical set of $c_1$ is unbounded above, or  the restriction of $-c_1$ to some interval of the form $\{t \geq b_{{\rm M},0}\} \subset \mathbb{R}$ is strongly decreasing.

	\end{enumerate}
Then, the Reeb space of the function ${\pi}_{m+1,1} {\mid}_{X_{m,c_1,c_2}}$ is regarded to be the Reeb graph and is isomorphic to some graph.
\end{Prop}
\begin{proof}
By the following exposition, our proof is essentially completed.

The assumptions (\ref{prop:2.1}--\ref{prop:2.4}) enable us to find paths like the path in the proof of Proposition \ref{prop:1} departing from $p_{D}$ obtained by joining the set $\{(t,p_{D,2}) \mid p_{D,1} \leq t \leq c_2(p_{D,2})\}$ ($\{(t,p_{D,2}) \mid c_2(p_{D,2}) \leq t \leq p_{D,1}\}$), which is a one-point set or a segment, and the path in $\overline{{D}_{m,c_1,c_2}}-{D}_{m,c_1,c_2}$ such that the restriction of ${\pi}_{2,1}$ to the resulting path is injective and that
the point $(p_{s,c_2},{(0)}_{j=1}^{m-1}) \in {\mathbb{R}}^2 \times {\mathbb{R}}^{m-1}={\mathbb{R}}^{m+1}$ is a critical point of ${\pi}_{m+1,1} {\mid}_{X_{m,c_1,c_2}}$.

\end{proof}
Note that this is an extended version of Proposition \ref{prop:1}.
\begin{Prop}
	\label{prop:3}
	Let $c_1$ and $c_2$ be functions of Theorem \ref{thm:2} {\rm (}\ref{thm:2.2}{\rm )}. We abuse the notation. Suppose also that at least one of the following holds.
	\begin{enumerate}
		\item \label{prop:3.1} The critical set of $c_2$ is bounded below, the restriction of $c_2$ to each interval of the form $\{t \leq b_{\rm m}\} \subset \mathbb{R}$ is not strongly increasing, and the critical set of $c_1$ is bounded below. 
\item  \label{prop:3.2} The critical set of $c_2$ is bounded above, the restriction of $c_2$ to each interval of the form $\{t \geq b_{\rm M}\} \subset \mathbb{R}$ is not strongly decreasing, and the critical set of $c_1$ is bounded above. 
		\item \label{prop:3.3} The critical set of $c_1$ is bounded below, the restriction of $-c_1$ to each interval of the form $\{t \leq b_{\rm m}\} \subset \mathbb{R}$ is not strongly increasing, and the critical set of $c_2$ is bounded below. 
	
			\item \label{prop:3.4} The critical set of $c_1$ is bounded above, the restriction of $-c_1$ to each interval of the form $\{t \geq b_{\rm M}\} \subset \mathbb{R}$ is not strongly decreasing, and the critical set of $c_2$ is bounded above.

	\end{enumerate}
Then, the Reeb space of the function ${\pi}_{m+1,1} {\mid}_{X_{m,c_1,c_2}}$ is regarded to be the Reeb graph and is not isomorphic to any finite or infinite graph.
\end{Prop}
\begin{proof}
The Reeb space of the function ${\pi}_{m+1,1} {\mid}_{X_{m,c_1,c_2}}$ is regarded to be the Reeb graph from Theorem \ref{thm:2}.

Suppose that (\ref{prop:3.1}) holds. 

In this case, the critical set of $c_2$ is bounded below and the restriction of $c_2$ to each interval of the form $\{t \leq b_{\rm m}\} \subset \mathbb{R}$ is not strongly increasing. We can see that the restriction of $c_2$ to some interval $\{t \leq b_{{\rm m},0}\} \subset \mathbb{R}$ is strongly decreasing and that there exist no critical point of $c_1$ on $\{t \leq b_{{\rm m},0}\} \subset \mathbb{R}$.
We can also see that for $i=1,2$, the restriction of $c_i$ to the interval $\{t \leq b_{{\rm m},0}\} \subset \mathbb{R}$ is strongly decreasing.

The set $\{(x_1,x_2) \mid x_2 \leq b_{{\rm m},0}, c_2(b_{{\rm m},0}) \leq x_1 \leq c_2(x_2), c_1(x_2) \leq x_1 \leq c_2(x_2) \}$ is connected and in $\overline{D_{c_1,c_2}}$, and contains no critical point of $c_i$ for $i=1,2$. From this local property, we cannot construct paths as in the proofs of Propositions \ref{prop:1} and \ref{prop:2} and we can see that the Reeb space of the function ${\pi}_{m+1,1} {\mid}_{X_{m,c_1,c_2}}$ is regarded to be the Reeb graph and that it is not isomorphic to any finite or infinite graph.

By a kind of symmetry, we can discuss similarly in the remaining three cases (\ref{prop:3.2}--\ref{prop:3.4}).

This completes the proof.
\end{proof}

The following is one of our main result.
\begin{Thm}
\label{thm:3}
Let $c_1$ and $c_2$ be functions of Theorem \ref{thm:2} {\rm (}\ref{thm:2.2}{\rm )}. We abuse the notation.

The Reeb space of the function ${\pi}_{m+1,1} {\mid}_{X_{m,c_1,c_2}}$ is regarded to be the Reeb graph. It is isomorphic to some graph if and only if the following four hold.

\begin{enumerate}
\item \label{thm:3.1} Either the condition {\rm (}\ref{prop:2.1}{\rm )} of Proposition \ref{prop:2} holds, or in the condition {\rm (}\ref{prop:3.1}{\rm )} of Proposition \ref{prop:3}, the critical set of $c_2$ is bounded below, the restriction of $c_2$ to each interval of the form $\{t \leq b_{\rm m}\} \subset \mathbb{R}$ is not strongly increasing, and instead, the critical set of $c_1$ is unbounded below.  
\item \label{thm:3.2} Either the condition {\rm (}\ref{prop:2.2}{\rm )} of Proposition \ref{prop:2} holds, or in the condition {\rm (}\ref{prop:3.2}{\rm )} of Proposition \ref{prop:3}, the critical set of $c_2$ is bounded above, the restriction of $c_2$ to each interval of the form $\{t \geq b_{\rm M}\} \subset \mathbb{R}$ is not strongly decreasing, and instead, the critical set of $c_1$ is unbounded above. 
\item \label{thm:3.3} Either the condition {\rm (}\ref{prop:2.3}{\rm )} of Proposition \ref{prop:2} holds, or in the condition {\rm (}\ref{prop:3.3}{\rm )} of Proposition \ref{prop:3}, the critical set of $c_1$ is bounded below, the restriction of $-c_1$ to each interval of the form $\{t \leq b_{\rm m}\} \subset \mathbb{R}$ is not strongly increasing, and instead, the critical set of $c_2$ is unbounded below. 
\item \label{thm:3.4} Either the condition {\rm (}\ref{prop:2.4}{\rm )} of Proposition \ref{prop:2} holds, or in the condition {\rm (}\ref{prop:3.4}{\rm )} of Proposition \ref{prop:3}, the critical set of $c_1$ is bounded above, the restriction of $-c_1$ to each interval of the form $\{t \geq b_{\rm M}\} \subset \mathbb{R}$ is not strongly decreasing, and instead, the critical set of $c_2$ is unbounded above. 
\end{enumerate}
\end{Thm}
\begin{proof}
The Reeb space of the function ${\pi}_{m+1,1} {\mid}_{X_{m,c_1,c_2}}$ is regarded to be the Reeb graph of course.

We prove the latter statement.

We prove the sufficiency.

We consider the condition (\ref{thm:3.1}). In the condition {\rm (}\ref{prop:3.1}{\rm )} of Proposition \ref{prop:3}, suppose that the critical set of $c_2$ is bounded below, that the restriction of $c_2$ to each interval of the form $\{t \leq b_{\rm m}\} \subset \mathbb{R}$ is not strongly increasing, and that the critical set of $c_1$ is unbounded below. We abuse the notation of Proposition \ref{prop:3} and its proof. The set $\{(x_1,x_2) \mid x_2 \leq b_{{\rm m},0}, c_2(b_{{\rm m},0}) \leq x_1 \leq c_2(x_2), c_1(x_2) \leq x_1 \leq c_2(x_2) \}$ is connected and in $\overline{D_{c_1,c_2}}$, and contains critical points of $c_1$, which is also unbounded below.

Hereafter, we also abuse some notation from Proposition \ref{prop:1} and its proof.
We choose a point $p$ of $X_{m,c_1,c_2}$ mapped to a point $p_e$ in some edge $e$ of the Reeb graph by the quotient map and a point $p_{D}:={\pi}_{2,1}(p) \in \{(x_1,x_2) \mid x_2 \leq b_{{\rm m},0}, c_2(b_{{\rm m},0}) \leq x_1 \leq c_2(x_2), c_1(x_2) \leq x_1 \leq c_2(x_2) \}$. Let $p_{D}:=(p_{D,1},p_{D,2})$. We consider a point $p_{D,c_1,q}:=(p_{D,1},q_{D,1})$ with $c_1(q_{D,1})=p_{D,1} \leq p_{D,2}$ and we can have such a point due to our assumption that the values $c_1(t)$ and $c_2(t)$ converge to $q_1$ as $t$ diverges to $-\infty$. The set $\{(p_{D,1},t) \mid q_{D,1} \leq t \leq p_{D,2}\}$ is in $\{(x_1,x_2) \mid x_2 \leq b_{{\rm m},0}, c_2(b_{{\rm m},0}) \leq x_1 \leq c_2(x_2), c_1(x_2) \leq x_1 \leq c_2(x_2) \}$. We can consider some path from $p_{D,c_1,q}$ to another point $p_{s,c_1,q}$ in $\{(x_1,x_2) \mid x_2<b_{{\rm m},0}, c_1(x_2)=x_1\}$ and have the path departing from $p_{D}$ obtained by joining the set $\{(p_{D,1},t) \mid q_{D,1} \leq t \leq p_{D,2}\}$, which is a one-point set or a segment, and the path in $\{(x_1,x_2) \mid x_2 \leq b_{{\rm m},0}, c_2(b_{{\rm m},0}) \leq x_1 \leq c_2(x_2), c_1(x_2) \leq x_1 \leq c_2(x_2) \}$ in such a way that the restriction of ${\pi}_{2,1}$ to the latter path is injective and that the point $(p_{s,c_1,q},{(0)}_{j=1}^{m-1}) \in {{\mathbb{R}}}^2 \times {\mathbb{R}}^{m-1}={\mathbb{R}}^{m+1}$ is a critical point of ${\pi}_{m+1,1} {\mid}_{X_{m,c_1,c_2}}$. 

We discuss a fact obtained from the argument on the assumption {\rm (}\ref{prop:3.1}{\rm )} of Proposition \ref{prop:3} here. 
First, the assumption {\rm (}\ref{prop:3.2}{\rm )} of Proposition \ref{prop:3} is essentially same as {\rm (}\ref{prop:3.1}{\rm )} of Proposition \ref{prop:3}. We remember some arguments in the proof of Proposition \ref{prop:1} (and Proposition \ref{prop:2}) in the following arguments. From the assumptions (\ref{thm:3.1}, \ref{thm:3.2}) of the present theorem, for a (more) general point $p$ of $X_{m,c_1,c_2}$ mapped to a point $p_e$ in some edge $e$ of the Reeb graph by the quotient map and a point $p_{D}:={\pi}_{2,1}(p) \in \overline{{D}_{m,c_1,c_2}}$, we can have an oriented path from $p_e$ to a vertex of the Reeb graph the orientation at each point of which is same as the orientation induced from the Reeb digraph in the following way. We can choose finitely many sets of the form $\{(t,{p_{D,2}}^{\prime}) \mid {p_{D,1}}^{\prime} \leq t \leq c_2({p_{D,2}}^{\prime})\}$ containing a point ${p_{D}}^{\prime}:=({p_{D,1}}^{\prime},{p_{D,2}}^{\prime})$ or $\{({p_{D,1}}^{\prime},t) \mid {q_{D,1}}^{\prime} \leq t \leq {p_{D,2}}^{\prime}\}$ containing a point ${p_{D}}^{\prime}:=({p_{D,1}}^{\prime},{p_{D,2}}^{\prime})$ with the relation ${p_{D,1}}^{\prime}:=c_1({q_{D,1}}^{\prime})$ in $\overline{{D}_{m,c_1,c_2}}$ and a curve in the graph $\{(c_i(x),x) \mid x \in \mathbb{R}\}$ entering a point $p_{s,c}$ making $(p_{s,c},{(0)}_{j=1}^{m-1}) \in {\mathbb{R}}^2 \times {\mathbb{R}}^{m-1}={\mathbb{R}}^{m+1}$ a critical point of ${\pi}_{m+1,1} {\mid}_{X_{m,c_1,c_2}}$ and as a result we can have a path in $\overline{{D}_{m,c_1,c_2}}$ mapped onto a desired path connecting $p_e$ and a desired vertex in the Reeb graph by the quotient map $q_{{\pi}_{2,1} {\mid}_{\overline{D_{c_1,c_2}}}}:\overline{D_{c_1,c_2}} \rightarrow R_{{\pi}_{2,1} {\mid}_{\overline{D_{c_1,c_2}}}}$.  The assumptions (\ref{thm:3.3}, \ref{thm:3.4}) of the present theorem are regarded to be a kind of duals to the assumptions (\ref{thm:3.1}, \ref{thm:3.2}) of the present theorem. We can argue as in the proofs of Propositions \ref{prop:1} and \ref{prop:2} to complete the proof of the sufficiency.

We discuss the necessity. The necessity is due to Proposition \ref{prop:3}, its proof, and the related arguments above. This completes the proof of the necessity.

This completes the proof. 

\end{proof}
Theorem \ref{thm:3} is also regarded as an extension of Proposition \ref{prop:1}. Propositions \ref{prop:2} and \ref{prop:3} with their proofs are important in Theorem \ref{thm:3} and its proof. 
We only show an example for Proposition \ref{prop:3}.
\begin{Ex}\label{ex:2}
The case $c_1(x):=\frac{t_0}{x^2+1}$ ($0<t_0<1$) and $c_2(x):=\frac{1}{x^2+1}$ is an explicit case of Proposition \ref{prop:3} with the critical sets of $c_1$ and $c_2$ being bounded and one-point sets and $(q_1,q_2)=(0,0)$. 
If we consider a weaker condition $t_0<1$, then we can also have the Reeb digraph with exactly one edge, exactly two vertices, and the closure of the edge being homeomorphic to $D^1$, and in the case $t_0 \leq 0$, the assumption (\ref{thm:2.2.3}) of Theorem \ref{thm:2} (Theorem \ref{thm:2} (\ref{thm:2.2})) is also dropped. 
	\end{Ex}
\begin{Thm}
\label{thm:4}
In Theorem \ref{thm:1}, suppose that the set $c_1(S(c_1)) \bigcup c_2(S(c_2))$ is a discrete and closed subset of $\mathbb{R}$, that at least one function of $c_1$ and $c_2$ has at least one critical point and that both the values $c_1(t)$ and $c_2(t)$ diverge to $+\infty$ or $-\infty$ as $t$ diverges to $\pm \infty$.
Then the Reeb space of the function ${\pi}_{m+1,1} {\mid}_{X_{m,c_1,c_2}}$ is regarded to be the Reeb graph, where we abuse the notation. In addition, the Reeb graph is isomorphic to a graph if and only if the set $S(c_1) \bigcup S(c_2) \subset \mathbb{R}$ is unbounded.	
\end{Thm}
\begin{proof}
By the structures of the functions and the maps, the function ${\pi}_{m+1,1} {\mid}_{X_{m,c_1,c_2}}$ is proper.

A proof of the former part of this theorem is shown like Theorem \ref{thm:2} (\ref{thm:2.2}). 

We prove the latter part.  \\
\ \\
Case 4-1 The case $c_1(t)$ and $c_2(t)$ go to $+\infty$ as $t$ diverges to $\pm \infty$. \\
This case contains the case $c_1(t)$ and $c_2(t)$ go to $-\infty$ as $t$ diverges to $\pm \infty$. In fact, it is sufficient to reverse the signs of the functions.

We discuss the sufficiency.

In the case the set $S(c_1) \bigcup S(c_2) \subset \mathbb{R}$ is unbounded, we can construct paths like the path in the proof of Theorem \ref{thm:3}. By abusing the notation and referring to the related part of the proof again, we explain the path. The path starts from $p_{D}$ and it is obtained by joining the set $\{(p_{D,1},t) \mid q_{D,1} \leq t \leq p_{D,2}\}$, which is a one-point set or a segment, and the path in $\{(x_1,x_2) \mid x_2 \leq b_{{\rm m},0}, c_2(b_{{\rm m},0}) \leq x_1 \leq c_2(x_2), c_1(x_2) \leq x_1 \leq c_2(x_2) \}$ such that the restriction of ${\pi}_{2,1}$ to the latter path is injective and that the point $(p_{s,c_1,q},{(0)}_{j=1}^{m-1}) \in {{\mathbb{R}}}^2 \times {\mathbb{R}}^{m-1}={\mathbb{R}}^{m+1}$ is a critical point of ${\pi}_{m+1,1} {\mid}_{X_{m,c_1,c_2}}$. 
For a general point $p$ of $X_{m,c_1,c_2}$ mapped to a point $p_e$ in some edge $e$ of the Reeb graph by the quotient map and a point $p_{D}:={\pi}_{2,1}(p) \in \overline{{D}_{m,c_1,c_2}}$, we can have an oriented path from $p_e$ to a vertex of the Reeb graph the orientation at each point of which is same as the orientation induced from the Reeb digraph, as we first argue in the proof of Proposition \ref{prop:1} (and Proposition \ref{prop:2}).

The functions $c_1$ and $c_2$ are bounded below and due to this together with the arguments in the proofs of some Propositions and Theorems above, we have another oriented path from $p_e$ to a vertex of the Reeb graph the orientation at each point of which is different from the orientation induced from the Reeb digraph.

This completes the proof of the sufficiency in this case.

The necessity is, as in the proof of Theorem \ref{thm:3}, essentially due to the proof of Proposition \ref{prop:3}. For this, we also remember the behaviors of the functions $c_1$ and $c_2$ at each infinity. This completes the proof of the necessity in this case.

\ \\
Case 4-2 The case both $c_1(t)$ and $c_2(t)$ go to $-\infty$ ($+\infty$) as $t$ diverges to $-\infty$ (resp. $+\infty$) . \\
This case contains the case $c_1(t)$ and $c_2(t)$ go to $+\infty$ ($-\infty$) as $t$ diverges to $-\infty$ (resp. $+\infty$). In fact, it is sufficient to reverse the signs of the functions, as in Case 4-1.

The necessity, is, as in Case 4-1. We discuss the sufficiency. For this, due to the assumptions on the behaviors of $c_1$ and $c_2$ at each infinity, we can argue as Case 4-1 and complete the proof of the sufficiency in Case 4-2. \\
\ \\
This completes the proof.
\end{proof}
\begin{Ex}

\label{ex:3}
We review Example \ref{ex:1}. We have the first derivative of $c_1$ as ${c_1}^{\prime}(x)=\frac{2x}{{(x^2+1)}^2}{(\sin x)}^2+\frac{2\sin x \cos x}{x^2+1}$, which is bounded and converges to $0$ at each infinity. We have the first derivative of $c_2$ as ${c_2}^{\prime}(x)=\frac{2x}{{(x^2+1)}^2}$, which is bounded and converges to $0$ at each infinity.
By rotating the two graphs in each case from Example \ref{ex:1} (\ref{ex:1.1}, \ref{ex:1.2}) simultaneously, slightly, and suitably, we can easily have a case for Theorem \ref{thm:4} with the following conditions.
\begin{itemize}
\item The resulting Reeb graph is not isomorphic to any graph. 
\item The resulting case is for Case 4-2 in (the proof of) Theorem \ref{thm:4}.
\end{itemize}
\end{Ex}

\begin{Thm}
\label{thm:5}
In Theorem \ref{thm:1}, suppose that the intersection $(\mathbb{R}-\{q\}) \bigcap (c_1(S(c_1)) \bigcup c_2(S(c_2)))$ is a discrete and closed subset of $\mathbb{R}-\{q\}$ and that at least one function of $c_1$ and $c_2$ has at least one critical point. Suppose also the following.
			
			\begin{enumerate}
				\item \label{thm:5.1}
				As $t \in \mathbb{R}$ diverges to $-\infty$ {\rm (}$+\infty${\rm )}, then both the values $c_1(t)$ and $c_2(t)$ converge to $q \in \mathbb{R}$.
				\item \label{thm:5.2} As $t \in \mathbb{R}$ diverges to $+\infty$  {\rm (}resp. $-\infty${\rm )}, then both the values $c_1(t)$ and $c_2(t)$ go to $+\infty$ or both the values go to $-\infty$.
				
				\item \label{thm:5.3} Each connected component $K_{c_1,c_2,q,j}$ of the set ${{\pi}_{m+1,1}}^{-1}(q) \bigcap X_{m,c_1,c_2}$ has a small compact and connected neighborhood $N_{c_1,c_2,q,j}$ in $X_{m,c_1,c_2}$ consisting of contours of the the restriction ${\pi}_{m+1,1} {\mid}_{X_{m,c_1,c_2}}$ and exactly one contour $K_{c_1,c_2,q,j}$ in a level set of the form ${{\pi}_{m+1,1}}^{-1}(q) \bigcap X_{m,c_1,c_2}$. The complementary set of $K_{c_1,c_2,q,j}$ in the connected component $N_{c_1,c_2,q,j}$ consists of finitely many connected sets and contains no critical contour of ${\pi}_{m+1,1} {\mid}_{X_{m,c_1,c_2}}$. In addition, these neighborhoods $N_{c_1,c_2,q,j}$ can be chosen to be mutually disjoint.
				
			\end{enumerate}
		Then the Reeb space of the function ${\pi}_{m+1,1} {\mid}_{X_{m,c_1,c_2}}$ is regarded to be the Reeb graph. Furthermore, the Reeb graph is a graph if and only if the following two hold.

\begin{enumerate}
\setcounter{enumi}{3}
\item \label{thm:5.4}
The assumptions of {\rm (}\ref{thm:3.1}, \ref{thm:3.3}{\rm )} {\rm (}resp. {\rm (}\ref{thm:3.2}, \ref{thm:3.4}{\rm )}{\rm )} of Theorem \ref{thm:3} are satisfied with $q=q_1$ {\rm (}resp. $q=q_2${\rm )}. Note that such conditions can be discussed naturally in the present situation.
\item \label{thm:5.5} The set $S(c_1) \bigcup S(c_2)$ is unbounded above {\rm (}resp. below{\rm )}.
\end{enumerate}
\end{Thm}

\begin{proof}
A proof of the former part of this theorem is shown as in the proof of Theorem \ref{thm:2} (\ref{thm:2.2}) and Theorem \ref{thm:4}. 

The latter part is shown as in the proofs of several Propositions and Theorems above. As in these proofs, we discuss whether for an arbitrary point $p_e$ in an arbitrary edge $e$ of the Reeb graph, we can have a path passing the point $p_e$, containing $e$, connecting some two vertices in the Reeb graph, and having a consistent orientation induced from the Reeb digraph.

This completes the proof.
\end{proof}
\begin{Ex}
	\label{ex:4}
	We can find a non-negative smooth function $c_0(x):\mathbb{R} \rightarrow \mathbb{R}$ with $c_0(x):=0$ ($x \leq 0$), $c_0(x):=x+2\sin x$ ($x>T_0$ with a suitable $T_0>0$) and $c_0(x_1)<c_0(x_2)$ for an arbitrary pair of real numbers with $0 \leq x_1<x_2 \leq T_0$. Let $c_1(x):=\frac{1}{2(x^2+1)}({\sin}^2 x)+c_0(x)$ and $c_2(x):=\frac{1}{x^2+1}+c_0(x)$. 
	
This is a case for Theorem \ref{thm:5} with the set $S(c_1) \bigcup S(c_2) \subset \mathbb{R}$ being unbounded and $q=0$. This is also a case for Theorem \ref{thm:5} satisfying the assumptions (\ref{thm:5.4}, \ref{thm:5.5}). 
\end{Ex}
The following may not be so difficult. The author does not reach explicit answers yet.
\begin{Prob}
	Can we find examples for Theorem \ref{thm:5} in the real analytic category and explicit single elementary functions $c_1$ and $c_2$?
	\end{Prob}
\begin{Prob}
Find further examples for cases of Propositions and Theorems we present here. In several cases, examples are not given in the present paper. Can we find them in the real analytic situations or using single elementary functions?
\end{Prob}
	
 \section{Conflict of interest and Data availability.}
  \noindent {\bf Conflict of interest.} \\
 The author is a researcher at Osaka Central Advanced Mathematical Institute (OCAMI researcher). This is supported by MEXT Promotion of Distinctive Joint Research Center Program JPMXP0723833165. He thanks this, where he is not employed there. \\
  %Some of works by other researchers and this version may overlap in some of the contents due to the nature that our problems are natural in theory of Morse functions and applications to differential topology and that related mathematical studies are very fundamental and classical in some senses, for example. However the present version of our paper is presented independent of these work. \\
  %Saga Souhatsu Mathematical Seminar (http://inasa.ms.saga-u.ac.jp/Japanese/saga-souhatsu.html), inviting the author as a speaker, is funded and supported by JST Fusion Oriented REsearch for disruptive Science and Technology JPMJFR202U: the author was a speaker on 2024/7/12 supported by this project.\\
  \ \\
  {\bf Data availability.} \\
  No data other than the present article is generated. We do not assume non-trivial arguments in the preprints (of the author) which are still unpublished formally, where we may refer to the preprints.


\begin{thebibliography}{25}
%	\bibitem{buchstaberpanov} V. M. Buchstaber and T. E. Panov, \textsl{Toric topology}, Mathematical Surveys and Monographs, Vol. 204, American Mathematical Society, Providence, RI, 2015.

	%	\bibitem{calabi} E. Calabi, Quasi-surjective mappings and a generation of Morse theory, Proc. U.S.-Japan Seminar in Differential Geometry, Kyoto, 1965, pp. 13--16.
	%
	%		\bibitem{cavicchioli} A. Cavicchioli, \textsl{Covering numbers of manifolds and critical points of a Morse function}, Israel. J. Math. 70 (1990), 279--304.
%	\bibitem{cerf} J. Cerf, \textsl{La stratification naturelle des espaces de fonctions deff\'erentiables r'eelles et le th'eor`eme de la pseudo-isotopie}, Inst. Hautes Etudes Sci. Publ. Math. 39 (1970), 5--173.
	%		\bibitem{choimasudasuh} S. Choi, M. Masuda and D. Y. Suh, \textsl{Topological classification of generalized Bott towers}, Trans. Amer. Math. Soc. 362 (2010), 1097--1112.

	%		\bibitem{cornealuptonopreatanre} O. Cornea, G. Lupton, J. Oprea and D. Tanr\'e, \textsl{Lusternik-Schnirelmann category}, Mathematical Surveys and Monographs, 103, Amer. Math. Soc., Providence, RI, 2003.
	%\bibitem{crowleyescher} D. Crowley and C. Escher, \textsl{A classification of $S^3$-bundles over $S^4$}, Differential. Geom. Appl. 18 (2003), 363--380, arXiv:0004147.
	%\bibitem{crowleynordstrom} D. Crowley and J. Nordstr\"{o}m, \textsl{The classification of $2$-connected $7$-manifolds}, Proc. London. Math. Soc. 119 (2019), 1--54, arXiv:1406.2226.

%\bibitem{bochnakcosteroy} J. Bochnak, M. Coste and M.-F. Roy, \textsl{Real algebraic geometry}, Ergebnisse der Mathematik und ihrer Grenzgebiete (3) [Results in Mathematics and Related Areas (3)], vol. 36, Springer-Verlag, Berlin, 1998. Translated from the 1987 French original; Revised by the authors.
%0		\bibitem{bochnakkucharz} J. Bochnak and W. Kucharz, \textsl{Algebraic approximation of mappings into spheres}, Michigan Mathematical Journal, vol. 34, no. 1, 1987.
%	\bibitem{bodinpopescupampusorea} A. Bodin, P. Popescu-Pampu and M. S. Sorea, \textsl{Poincar\'e-Reeb graphs of real algebraic domains}, Revista Matem\'atica Complutense, https://link.springer.com/article/10.1007/s13163-023-00469-y, arXiv:2207.06871v2.
%\bibitem{bott} R. Bott, \textsl{Nondegenerate critical manifolds}, Ann. of Math. 60 (1954), 248--261.
%	\bibitem{burletderham} O. Burlet and G. de Rham, \textsl{Sur certaines applications g\'en\'eriques d'une vari\'et\'e close a $3$ dimensions dans le plan}, Enseign. Math. 20 (1974). 275--292.
%\bibitem{costantino}  F. Costantino, \textsl{A short introduction to shadows of $4$-manifolds}, Fundamenta Mathematicae 251 no. 2 (2005), 427--442.
%\bibitem{costantinothurston} F. Costantino, D. Thurston, \textsl{$3$-manifolds efficiently bound $4$-manifolds}, J. Topol. 1 (2008),
%703--745.
%	\bibitem{delzant} T. Delzant, \textsl{Hamiltoniens p\'eriodiques et images convexes de l'application moment}, Bull. Soc. Math. France 116 (1988), No. 3, 315--339.
%\bibitem{ehresmann} C. Ehresmann, \textsl{Les connexions infinitesimales dans un espace fibre differentiable}, Colloque de Topologie, Bruxelles (1950), 29--55.

%\bibitem{eliashberg1} Y. Eliashberg, \textsl{On singularities of folding type}, Math. USSR Izv. 4 (1970). 1119--1134.
%\bibitem{eliashberg2} Y. Eliashberg, \textsl{Surgery of singularities of smooth mappings}, Math. USSR Izv. 6 (1972). 1302--1326.
%\bibitem{elredge} N. Elredge, \textsl{{\it Answer to} On finding polynomials that approximate a function and its derivative}, StackExchange, question 555712 (2013), https://math.stackexchange.com/questions/555712/on-finding-polynomials-that-approximate-a-function-and-its-derivative-extension.
%\bibitem{fernandezMunoz} M. Fern\'andez and V. mu\~noz, \textsl{On non-formal simply connected manifolds}, Topology Appl. 135 Issues 1--3 (2004), 111--117. math.DG/0212141.
%\bibitem{furuyaporto} Y. K. S. Furuya and P. Porto, \textsl{On special generic maps from a closed manifold into the plane}, Topology Appl. 35 (1990), 41--52.
%\bibitem{fujitakitabeppumitsuishi} H. Fujita, Y Kitabeppu and A. Mitsuishi, \textsl{Distance functions and convex bodies and symplectic toric manifolds}, arXiv:2003.02293.

%\bibitem{gabard} A. Gabard, \textsl{Notes on Hilbert's 16th: experiencing Viro's theory}, arXiv:1310.1865, 2013.
\bibitem{gelbukh} I. Gelbukh, \textsl{On the topology of the Reeb graph}, Publicationes Mathematicae Debrecen 104(3--4) (2023), 343--365.
%\bibitem{gelbukh} I. Gelbukh, \textsl{Loops in Reeb graphs of $n$-manifolds}, diskrete \& Computational Geometry, 59 (4) (2018), 843--863. 


%\bibitem{gelbukh2} I. Gelbukh, \textsl{Realization of a digraph as the Reeb graph of a Morse-Bott function on a given surface}, Topology and its Applications, 2024. 
%\bibitem{gelbukh3} I. Gelbukh, \textsl{Reeb Graphs of Morse-Bott Functions on a Given Surface}, Bulletin of the Iranian Mathematical Society, Volume 50 Article number 84, 2024, 1--17. 

%%\bibitem{gelbukh2} I. Gelbukh, \textsl{Approximation of Metric Spaces by Reeb Graphs: Cycle Rank of a Reeb Graph, the Co-rank of the Fundamental Group, and Large Components of Level Sets on Riemannian Manifolds}, Filomat (in press), arxiv:1903.00777.
%\bibitem{gelbukh} I. Gelbukh, \textsl{A finite graph is homeomorphic to the Reeb graph of a Morse-Bott function}, Mathematica Slovaca, 71 (3), 757--772, 2021; doi: 10.1515/ms-2021-0018. 
%%\bibitem{gelbukh2} I. Gelbukh, \textsl{Morse-Bott functions with two critical values on a surface}, Czechoslovak Mathematical Journal, 71 (3), 865--880, 2021; doi: 10.21136/CMJ.2021.0125-20. 
%\bibitem{golubitskyguillemin} M. Golubitsky and V. Guillemin, \textsl{Stable Mappings and Their Singularities}, Graduate Texts in Mathematics (14), Springer-Verlag(1974).
\bibitem{hatcher} A. Hatcher, \textsl{Algebraic Topology}, Electronic version, https://pi.math.cornell.edu/$\sim$hatcher/AT/AT+.pdf.
%\bibitem{hatcher} A. E. Hatcher, \textsl{A proof of the Smale conjecture}, Ann. of Math. 117 (1983), 553--607.
%\bibitem{hempel} J. Hempel, \textsl{3- Manifolds}, AMS Chelsea Publishing, 2004. 
%\bibitem{hiratukasaeki} J. T. Hiratuka and O. Saeki, \textsl{Triangulating Stein factorizations of generic maps and Euler Characteristic formulas}, RIMS Kokyuroku Bessatsu B38 (2013), 61--89. 
%\bibitem{hiratukasaeki2} J. T. Hiratuka and O. Saeki, \textsl{Connected components of regular fibers of differentiable maps}, in "Topics on Real and Complex Singularities", Proceedings of the 4th Japanese-Australian Workshop (JARCS4), Kobe 2011,  World Scientific, 2014, 61--73. 
%\bibitem{ishikawakoda} M. Ishikawa and Y. Koda, \textsl{Stable maps and branched shadows of $3$-manifolds}, Mathematische Annalen 367 (2017), no. 3, 1819--1863, arXiv:1403.0596.
\bibitem{izar} Izar. S. A, \textsl{Funções de Morse e Topologia das Superfícies I: O grafo de Reeb de $f:M \rightarrow \mathbb{R}$}, Métrica no. 31, In Estudo e Pesquisas em Matemática, Brazil: IBILCE, 1988, https://www.ibilce.unesp.br/Home/Departamentos/Matematica/metrica-31.pdf.
%\bibitem{kervairemilnor} M. Kervaire and J. W. Milnor, \textsl{Groups of homotopy spheres : I}, Ann. of Math. 77 (1963), 504--537.
%\bibitem{kitazawa1} N. Kitazawa, \textsl{On round fold maps} (in Japanese), RIMS Kokyuroku Bessatsu B38 (2013), 45--59.
%\bibitem{kitazawa2} N. Kitazawa, \textsl{On manifolds admitting fold maps with singular value sets of concentric spheres}, Doctoral Dissertation, Tokyo Institute of Technology (2014).
%\bibitem{kitazawa3} N. Kitazawa, \textsl{Fold maps with singular value sets of concentric spheres}, Hokkaido Mathematical Journal Vol.43, No.3 (2014), 327--359.
\bibitem{kitazawa1} N. Kitazawa, \textsl{On Reeb graphs induced from smooth functions on $3$-dimensional closed orientable manifolds with finitely many singular values}, Topol. Methods in Nonlinear Anal. Vol. 59 No. 2B, 897--912, arXiv:1902.08841.
\bibitem{kitazawa2} N. Kitazawa, \textsl{On Reeb graphs induced from smooth functions on closed or open surfaces}, Methods of Functional Analysis and Topology Vol. 28 No. 2 (2022), 127--143, arXiv:1908.04340.
\bibitem{kitazawa3} N. Kitazawa, \textsl{Real algebraic functions on closed manifolds whose Reeb graphs are given graphs}, Methods of Functional Analysis and Topology Vol. 28 No. 4 (2022), 302--308, arXiv:2302.02339, 2023. 
%\bibitem{kitazawa3} N. Kitazawa, \textsl{On Reeb graphs induced from smooth functions on $3$-dimensional closed manifolds which may not be orientable}, Methods of Functional Analysis and Topology Vol. 29 No. 1 (2023), 57--72, 2024.

\bibitem{kitazawa4} N. Kitazawa, \textsl{
Constructing Morse functions with given Reeb graphs and level sets}, accepted for publication in Topol. Methods in Nonlinear Anal., arXiv:2108.06913
(, where the title has been changed from the title there), 2025.
%\bibitem{kitazawa} N. Kitazawa, \textsl{Maps on manifolds onto graphs locally regarded as the quotient maps onto Reeb spaces of some differentiable maps and a new reconstruction problem}, arXiv:1909.10315.
\bibitem{kitazawa5} N. Kitazawa, \textsl{Reconstructing real algebraic maps locally like moment maps with prescribed images and compositions with the canonical projections to the $1$-dimensional real affine space}, arXiv:2303.10723.
\bibitem{kitazawa6} N. Kitazawa, \textsl{Some remark on real algebraic maps which are topologically special generic maps and generalize the canonical projections of the unit spheres}, arXiv:2312.10646, 2024.
%\bibitem{kitazawa3} N. Kitazawa, \textsl{Explicit construction of explicit real algebraic functions and real algebraic manifolds via Reeb graphs}, Algebraic and geometric methods of analysis 2023 “The book of abstracts”, 49—51, this is the abstract book of the conference "Algebraic and geometric methods of analysis 2023" (https://www.imath.kiev.ua/$\sim$topology/conf/agma2023/), https://imath.kiev.ua/$\sim$topology/conf/agma2023/contents/abstracts/texts/kitazawa/kitazawa.pdf.
%\bibitem{kitazawa4} N. Kitazawa, \textsl{Closed manifolds admitting no special generic maps whose codimensions are negative and their cohomology rings}, submitted to a refereed journal, arxiv:2008.04226v5.
%\bibitem{kitazawa4} N. Kitazawa, \textsl{Notes on explicit special generic maps into Euclidean spaces whose dimensions are greater than $4$}, a revised version is submitted based on positive comments (major revision) by referees and editors after the first submission to a refereed journal, arxiv:2010.10078.
%\bibitem{kitazawa5} N. Kitazawa, \textsl{Construction of real algebraic functions with prescribed preimages}, a positive report for major revision has been sent and this is submitted again to a refereed journal, arXiv:2303.00953v3.%
%\bibitem{kitazawa6} N. Kitazawa, \textsl{A note on real algebraic maps which are topologically special generic maps}, previous version(s) of the present article, arXiv:2312.10646.
\bibitem{kitazawa7} N. Kitazawa, \textsl{A note on Reeb spaces of explicit real analytic functions}, submitted to a refereed journal, arXiv:2601.11648.
\bibitem{kitazawa8} N. Kitazawa, \textsl{Reeb spaces of functions being analytic on dense subsets and their graph structures}, https://arxiv.org/.
%\bibitem{kitazawa9} N. Kitazawa, \textsl{A note on cohomological structures of special generic maps}, a revised version is submitted based on positive comments by a referee and editors after the third submission to a refereed journal.
%\bibitem{kitazawa6} N. Kitazawa, \textsl{Round fold maps and the topologies and the differentiable structures of manifolds admitting explicit ones}, submitted to a refereed journal, arXiv:1304.0618.
%\bibitem{kitazawa7}  N. Kitazawa, \textsl{Round fold maps and the topologies and the differentiable structures of manifolds admitting explicit ones}, arXiv:1304.0618.

%\bibitem{kitazawa0.5} N. Kitazawa, \textsl{Constructing fold maps by surgery operations and homological information of their Reeb spaces}, submitted to a refereed journal, arxiv:1508.05630.
%\bibitem{kitazawa0.6} N. Kitazawa, \textsl{Notes on fold maps obtained by surgery operations and algebraic information of their Reeb spaces}, arxiv:1811.04080.

%\bibitem{kitazawa6} N. Kitazawa, \textsl{Notes on explicit special generic maps into Euclidean spaces whose dimensions are greater than $4$}, a revised version is submitted based on positive comments (major revision) by referees and editors after the first submission to a refereed journal, arxiv:2010.10078.
%\bibitem{kitazawa6} N. Kitazawa, \textsl{On Reeb graphs induced from smooth functions on $3$-dimensional closed manifolds which may not be orientable}, a revised version is submitted to a refereed journal after based on positive comments by editors and referees after the second submission to a refreed journal, arXiv:2108.01300.
%\bibitem{kitazawa2} N. Kitazawa, \textsl{Realization problems of graphs as Reeb graphs of Morse functions with prescribed preimages}, submitted to a refereed journal, arXiv:2108.06913.
%\bibitem{kitazawa10} N. Kitazawa,\textsl{Round fold maps on $3$-dimensional manifolds and their integral and rational cohomology rings}, arXiv:2301.07008.
%\bibitem{kitazawa7} N. Kitazawa, \textsl{A class of naturally generalized special generic maps}, arXiv:2212.03174.
%\bibitem{kitazawa8} N. Kitazawa, \textsl{A note on cohomological structures of special generic maps}, a revised version is submitted based on positive comments by referees and editors after the third submission to a refereed journal.
%\bibitem{kitazawasaeki1} N. Kitazawa and O. Saeki, \textsl{Round fold maps on $3$-manifolds}, accepted for publication after a refereeing process and to appear in Algebraic \& Geometric Topology, arXiv:2105.00974.
		%	\bibitem{kitazawasaeki2} N. Kitazawa and O. Saeki, \textsl{Round fold maps of $n$-dimensional manifolds into ${\mathbb{R}}^{n-1}$}, submitted to a refereed journal, arXiv:2111.13510.
%\bibitem{ishikawakoda} M. Ishikawa, Y. Koda, \textsl{Stable maps and branched shadows of $3$-manifolds}, arXiv:1403.0596.
%\bibitem{kobayashisaeki} M. Kobayashi and O. Saeki, \textsl{Simplifying stable mappings into the plane from a global viewpoint}, Trans. Amer. Math. Soc. 348 (1996), 2607--2636.
%\bibitem{kohnpieneranestadrydellshapirosinnsoreatelen} K. Kohn, R. Piene, K. Ranestad, F. Rydell, B. Shapiro, R. Sinn, M-S. Sorea and S. Telen, \textsl{Adjoints and Canonical Forms of Polypols}, arXiv:2108.11747.
%\bibitem{kollar} J. Koll\'ar, \textsl{Nash's work in algebraic geometry}, Bulletin (New Series) of the American Mathematical Society (2) 54, 2017, 307--324.
%\bibitem{kucharz} W. Kucharz, \textsl{Some open questions in real algebraic geometry}, Proyecciones Journal of Mathematics, Vol. 41 No. 2 (2022), Universidad Cat\'olica del Norte Antofagasta, Chile, 437--448.
%\bibitem{lellis} Camillo De Lellis, \textsl{The Masterpieces of John Forbes Nash Jr.}, The Abel Prize 2013--2017 (Helge Holden and Ragni Piene, eds.), Springer International Publishing, Cham, 2019, 391--499, https://www.math.ias.edu/delellis/sites/math.ias.edu.delellis/files/Nash\_Abel\_75.pdf, arXiv:1606.02551.
\bibitem{martinezalfaromezasarmientooliveira} J. Martinez-Alfaro, I. S. Meza-Sarmiento and R. Oliveira, \textsl{Topological  classification of simple Morse Bott functions on surfaces}, Contemp. Math. 675 (2016), 165--179.%
%\bibitem{marzantowiczmichalak} W. Marzantowicz and L. P. Michalak, \textsl{Relations between Reeb graphs, systems of hypersurfaces and epimorphisms onto free groups}, Fund. Math., 265 (2), 97--140, 2024.
\bibitem{masumotosaeki} Y. Masumoto and O. Saeki, \textsl{A smooth function on a manifold with given Reeb graph}, Kyushu J. Math. 65 (2011), 75--84.
%\bibitem{maciasvirgospereirasaez} E. Mac\'ias-Virg\'os and M. J. Pereira-S\'aez, Height functions on compact symmetric spaces, Monatshefte f\"ur Mathematik 177 (2015), 119--140.
%\bibitem{mather1} J. Mather, \textsl{Notes on topological stability}, Bull. Amer. Math. Soc. (N. S.) 49 (4) (2012), 475--506.
%\bibitem{mather2} J. N Mather, \textsl{Stability of $C^{\infty}$ mappings. I. The division theorem}, Ann. of Math. (2) 87, 89--104 (1968).
%%\bibitem{mather3} J. N. Mather, \textsl{Stability of $C^{\infty}$ mappings. III. Finitely determined mapgerms}, Inst. Hautes. \'Etudes Sci. Publ. Math. (35) (1968), 279--308.
%\bibitem{mather4} J. N. Mather, \textsl{Some non-finitely determined map-germs}, In: Symposia Mathematica, Vol. II (INDAM, Rome, 1968), 303--320, Academic Press, London (1969).
%.\bibitem{mather5} J. N. Mather, \textsl{Stability of $C^{\infty}$ mappings. II. Infinitesimal stability implies stability}, Ann. of Math. (2) 89 (1969), 254--291 .
%\bibitem{mather6} J. N. Mather, \textsl{Stability of $C^{\infty}$ mappings. IV. Classification of stable germs by $R$-algebras}, Inst. Hautes \'Etudes Sci. Publ. Math. (37) (1969), 223--248.
%\bibitem{mather7} J. N. Mather, \textsl{Stability of $C^{\infty}$ mappings. V. Transversality}, Advances in Math. 4 (1970), 301--336.
%\bibitem{mather8} J. N. Mather, \textsl{Stability of $C^{\infty}$ mappings. VI: The nice dimensions}, In: Proceedings of Liverpool Singularities-Symposium, I (1969/70), 207--253, Lecture Notes in Math., Vol. 192 (1971).
\bibitem{michalak} L. P. Michalak, \textsl{Realization of a graph as the Reeb graph of a Morse function on a manifold}. Topol. Methods in Nonlinear Anal. 52 (2) (2018), 749--762, arXiv:1805.06727.
%\bibitem{michalak2} L. P. Michalak, \textsl{Combinatorial modifications of Reeb graphs and the realization problem}, arxiv:1811.08031.
%\bibitem{milnor} J. Milnor, \textsl{Singular points of complex hypersurfaces}, Annals of Mathematics Studies, No. 61, Princeton University Press, Princeton, N. J.; University of Tokyo Press, Tokyo, 1968.
%\bibitem{milnor1} J. Milnor, \textsl{On manifolds homeomorphic to the $7$-sphere}, Ann. of Math. (2) 64 (1956), 399--405.
%%\bibitem{milnor2} J. Milnor, \textsl{Lectures on the h-cobordism theorem}, Math. Notes, Princeton Univ. Press, Princeton, N.J. 1965.
%\bibitem{milnor3} J. W. Milnor, \textsl{Topology from the Differentiable Viewpoint}, University Press of Virginia, Charlottesville, VA, 1965.
%\bibitem{moise} E. E. Moise, \textsl{Affine Structures in $3$-Manifold{\rm :} V. The Triangulation Theorem and Hauptvermutung}, Ann. of Math., Second Series, Vol. 56, No. 1 (1952), 96--114.
%\bibitem{morin} B. Morin, \textsl{Formes canoniques des singulariti\'{e}s d\'{}une application diff\'{e}rentiable}, C. E. Acad. Sci. Paris 260 (1965), 5662--5665, 6503--6506.
%\bibitem{nash} J. Nash, \textsl{Real algebraic manifolds}, Ann. of Math. (2) 56 (1952), 405--421.
%\bibitem{ranicki} A. Ranicki, \textsl{Algebraic and geometric surgery}, https://www.maths.ed.ac.uk/~v1ranick/books/surgery.pdf, 2002.
%\bibitem{ramanujam} S. Ramanujam, \textsl{Morse theory of certain symmetric spaces}, J. Diff. Geom. 3 (1969), 213--229.
\bibitem{reeb} G. Reeb, \textsl{Sur les points singuliers d\'{}une forme de Pfaff compl\'{e}tement int\`{e}grable ou d\'{}une fonction num\'{e}rique}, Comptes Rendus
 Hebdomadaires des S\'{e}ances de I\'{}Acad\'{e}mie des Sciences 222 (1946), 847--849.
%\bibitem{ruas} Maria Aparecida Soares Ruas, \textsl{Old and New Results on Density of Stable Mappings}, Handbook of Geometry and Topology of Singularities III (2022) (editted y Jos\'e Luis Cisneros-Molina, L\^e D\~ung Tr\'ang and Jos\'e Seade), 1--80, Springer.
%\bibitem{saeki1} O. Saeki, \textsl{Notes on the topology of folds}, J. Math. Soc. Japan Volume 44, Number 3 (1992), 551--566.
%\bibitem{saeki1} O. Saeki, \textsl{Topology of special generic maps of manifolds into Euclidean spaces}, Topology Appl. 49 (1993), 265--293, we can also find at "https://core.ac.uk/download/pdf/81973672.pdf" for example.
\bibitem{saeki1} O. Saeki, \textsl{Reeb spaces of smooth functions on manifolds}, International Mathematics Research Notices, maa301, Volume 2022, Issue 11, June 2022, 3740--3768, https://doi.org/10.1093/imrn/maa301, arXiv:2006.01689.
\bibitem{saeki2} O. Saeki, \textsl{Reeb spaces of smooth functions on manifolds II}, Res. Math. Sci. 11, article number 24 (2024), https://link.springer.com/article/10.1007/s40687-024-00436-z.
%\bibitem{saeki0.2} O. Saeki, \textsl{Topology of singular fibers of differentiable maps}, Lecture Notes in Math., Vol. 1854, Springer-Verlag, 2004. 
%\bibitem{saeki4} O. Saeki, \textsl{Morse functions with sphere fibers}, Hiroshima Math. J. Volume 36, Number 1 (2006),  141--170.
%\bibitem{saeki} O. Saeki, \textsl{Reeb spaces of smooth functions on manifolds}, International Mathematics Research Notices, maa301, Volume 2022, Issue 11, June 2022, https://doi.org/10.1093/imrn/maa301, arXiv:2006.01689.
%\bibitem{saekitakase} O. Saeki and M. Takase, \textsl{Desingularizing special generic maps}, Journal of G\"okova Geometry Topology (2013), 1--24.
%\bibitem{sakurai} S. Sakurai, Master Thesis, Kyushu. Univ..
% \bibitem{saekitakase} O. Saeki and M. Takase, \textsl{Desingularizing special generic maps}, Journal of Gokova Geometry Topology 7 (2013), 1--24.
%\bibitem{saeki2} O. Saeki, \textsl{Topology of special generic maps of manifolds into Euclidean spaces}, Topology Appl. 49 (1993), 265--293.
%\bibitem{saeki4} O. Saeki, \textsl{Singular fibers and $4$-dimensional cobordism group}, Pacific J. Math. 248 (2010), 233--256.
%\bibitem{saekisakuma} O. Saeki and K. Sakuma, \textsl{On special generic maps into ${\mathbb{R}}^3$}, Pacific J. Math. 184 (1998), 175--193.
%\bibitem{saekisuzuoka} O. Saeki and K. Suzuoka, \textsl{Generic smooth maps with sphere fibers} J. Math. Soc. Japan Volume 57, Number 3 (2005), 881--902.
\bibitem{sharko} V. Sharko, \textsl{About Kronrod-Reeb graph of a function on a manifold}, Methods of Functional Analysis and
 Topology 12 (2006), 389--396.
%\bibitem{shiota} M. Shiota, \textsl{Thom's conjecture on triangulations of maps}, Topology 39 (2000), 383--399.
%\bibitem{smale1} S. Smale, \textsl{Generalized Poincare's conjecture in dimensions greater than four}, Ann. of Math. (2) 74 (1961) 391--406.
%\bibitem{smale2} S. Smale, \textsl{On the structure of manifolds}, Amer. J. Math 84 (1962), 387--399.
%\bibitem{smale} S. Smale, \textsl{On the structure of manifolds}, Amer. J. Math 84 (1962), 387--399.
%\bibitem{sorea1} M. S. Sorea, \textsl{The shapes of level curves of real polynomials near strict local maxima},  Ph. D. Thesis, Universit\'e de Lille, Laboratoire Paul Painlev\'e, 2018.
%\bibitem{sorea2} M. S. Sorea, \textsl{Measuring the local non-convexity of real algebraic curves}, Journal of Symbolic Computation 109 (2022), 482--509.
%\bibitem{sorea1} M. S. Sorea, \textsl{The shapes of level curves of real polynomials near strict local maxima},  Ph. D. Thesis, Universit\'e de %Lille, Laboratoire Paul Painlev\'e, 2018.
%\bibitem{sorea2} M. S. Sorea, \textsl{Measuring the local non-convexity of real algebraic curves}, J. Symbolic Compute. 109 (2022), 482--509.
%\bibitem{stong} R. E. Stong, \textsl{Notes on cobordsm theory}, Princeton Universty Press, 1968.
%\bibitem{takeuchi} M. Takeuchi, \textsl{Nice functions on symmetric spaces}, Osaka. J. Mat. (2) Vol. 6 (1969), 283--289%
%\bibitem{thom} R. Thom, \textsl{Les singularites des applications differentiables}, Ann. Inst. Fourier (Grenoble) 6 (1955-56), 43--87.
%\bibitem{tognoli} A. Tognoli, \textsl{Su una congettura di Nash}, Ann. Scuola Norm. Sup. Pisa (3) 27 (1973), 167--185.
%%\bibitem{turaev} Vladimir G. Turaev, \textsl{Topology of shadows}, Preprint, 1991.
%\bibitem{viro} O. Viro, \textsl{Patchworking real algebraic varieties}, based on the first chapter of doctoral dissertation, arXiv:math/0611382.
%\bibitem{wall} C. T. C Wall, \textsl{Classification problems in differential topology -- {\rm I:} Classificationon handlebodies}, Topology 2 (1963), 253--261.
%\bibitem{wall2} C. T. C. Wall \textsl{Classification problems in differential topology -- {\rm II:} Diffeomorphismsof handlebodies}, Topology 2 (1963), 263--272.
%\bibitem{wall3} C. T. C. Wall, \textsl{Classification problems in differential topology -- {\rm Q:} Quadratic forms on finite groups and related topics}, Topology 2 (1963), 281--298.
%\bibitem{wall4} C. T. C. Wall, \textsl{Classification problems in differential topology -- {\rm III:} Applications to special cases}, Topology 3 (1965), 291--304.
%%\bibitem{wall5} C. T. C. Wall, \textsl{Classification problems in differential topology -- {\rm IV:} Thickenings}, Topology 5 (1966), 73--94.
%\bibitem{wall6} C. T. C. Wall, \textsl{Classification problems in differential topology -- {\rm VI:} Classification of |{\rm (}$s-1${\rm )}-connected {\rm (}$2s+1${\rm )}-manifolds}, Topology 6 (3) (1967), 273--296.
%\bibitem{whitney} H.  Whitney,  \textsl{On singularities of mappings of Euclidean spaces: I,  mappings of the plane into the plane},  Ann.  of Math.  62 (1955),  374--410. 

	
\end{thebibliography}
\end{document}